\documentclass[12pt, reqno]{amsart}
\allowdisplaybreaks[1]
\usepackage{amsmath}
\usepackage{amssymb}
\usepackage{amsfonts}
\usepackage{verbatim}
\usepackage{amsthm}
\usepackage[usenames]{color}
\usepackage{hyperref}
\usepackage{mathrsfs}
\makeindex

 \newtheorem{theorem}{Theorem}[section]
 \newtheorem{corollary}[theorem]{Corollary}

 \newtheorem{proposition}[theorem]{Proposition}

\newtheorem{problem}[theorem]{Problem}

\theoremstyle{definition}

\theoremstyle{remark}

\newtheorem{example}[theorem]{Example}
\newtheorem{fact*}{Fact}
\newcommand{\hilbert}{\mathcal{H}}

\newcommand{\DC}{\tilde{''}}

\newcommand{\BH}{\mathbf{B}(\mathcal{H})}

\newcommand{\til}{\raise.17ex\hbox{$\scriptstyle\mathtt{\sim}$}}

\newcommand\beq{\begin{equation}}

\newcommand\eeq{\end{equation}}

\newcommand{\bbm}{\left[ \begin{smallmatrix}}
\newcommand{\ebm}{\end{smallmatrix} \right]}
\newcommand{\bpm}{\left( \begin{matrix}}
\newcommand{\epm}{\end{matrix} \right)}
\numberwithin{equation}{section}

\newlength{\Mheight}
\newlength{\cwidth}

\newcommand{\dfn}[1]{{\bf #1}\index{#1}}

\title[Invariant structure preserving functions]{Invariant structure preserving functions and an Oka-Weil Kaplansky density type theorem}
\author{
J. E. Pascoe 
}
\date{\today}
\thanks{$\ddagger$ Research supported by National Science Foundation DMS grant \#1953963.}

\begin{document}

\begin{abstract}
	We develop the theory of invariant structure preserving and free functions on a general structured topological space.
	We show that an invariant structure preserving function is pointwise approximiable by the appropriate analog of polynomials in the strong topology and therefore a free  function.
	Moreover, if a domain of operators on a Hilbert space is polynomially convex,
	 the set of free functions satisfies a Oka-Weil Kaplansky density type theorem-- contractive functions can be approximated by contractive polynomials.
\end{abstract}
\maketitle

\tableofcontents

	\section{Introduction}

	We prove several results.
	%(Here we use the word \emph{complete} in the sense of ``complete positivity" or ``complete boundedness" from operator theory.)
	Here are some interesting consequences of those results:
	\begin{enumerate}
		\item Let $\mathcal{H}$ be a topological vector space. Let $Y:\mathcal{H}\rightarrow\mathcal{H}$ be continuous on $\mathcal{H}$ and let
		$\mathbb{X}$ be a collection of operators on $\mathcal{H}.$ Then,
		$Y$ is in the strong operator topology closure of the unital algebra generated by the elements of $\mathbb{X}$ if and only if 
		for every $n\in \mathbb{N},$ if a subspace $S \subseteq \mathcal{H}^n$ is invariant for
		$X^{\oplus n}$ for all $X \in \mathbb{X}$ then it is invariant for 
		 $Y^{\oplus n}.$
		Radjavi and Rosenthal \cite{RadRos} proved the above result in the case of $\mathcal{H}$ a Hilbert space and $Y$ was an operator.
		This is a special case of our equidynamical approximation theorem, which can be viewed as analogous to the double
		 commutatant theorem. See Theorem \ref{equidyn}.
		\item
		Let $F \subseteq \mathbb{R}$ contain arbitrarily large numbers.
		Let $f:F \rightarrow \mathbb{R}$ be a function such that if $\mathscr{C}$ is a closed convex cone in $\mathbb{R}^n$ and $x\in F^n$ such that 
		$(x_ic_i)_{i=1,\ldots, n} \in \mathscr{C}$ when $c\in \mathscr{C},$ then $(f(x_i)c_i)_{i=1,\ldots, n} \in \mathscr{C}.$
		Then, $f$ analytically continues to an entire function with nonnegative
		power series coefficients at $0.$ See Theorem \ref{equifun} in light of Example \ref{exampleschoenberg}.
		Similarly, if we replace cones by convex sets, the sum of the power series coefficients must in addition sum to $1.$
		See Example \ref{exampleschoenberg2}.
		The result for cones is essentially a poor version of Schoenberg's theorem \cite{schoenberg1942,Belton2019,Belton2019a,fitzy}:
		a function $f:\mathbb{R}\rightarrow \mathbb{R}$ such that for any positive semidefinite matrix $[m_{ij}]_{i,j}$
		we have that $[f(m_{ij})]_{i,j}$ is again positive semidefinite must be an entire function with positive power series coefficients.
		A point of contrast is that, although our assumption is stronger in that we require arbitrary cones, we do not require $F$ to be open. 
		Moreover, we can completely characterize the functions which preserve invariance of cones to some fixed level $N,$ which contrasts
		with the classical Schoenberg theory where such problems are open. See Example \ref{exampleschoenberg3}.
		
		Thus, the following problem (which we do not solve here) is motivated: \emph{Give an elementary proof that positivity preservation in the Schoenberg sense implies
		invariant cone preservation in our sense.} That is, does the Schoenberg theorem require hard analysis?
		(For example, every proof known to the author involves delicately taking derivatives and then applying some kind of Bernstein type theorem.)
		Does the Schoenberg theorem hold on discrete sets, such as $\mathbb{Z}?$
		
		\item A free noncommutative function on a large enough domain has values contained in closure of the algebra generated by the coordinates of the input.
		See Theorem \ref{mainresult}.
		\item A contractive free noncommutative function on a polynomially convex domain can be approximated by contractive polynomials.
		Importantly, we eliminate the hypothesis that the domain contain a scalar or matricial point as was assumed in \cite{AGMCOP, BMV16}.
		See Theorem \ref{biggerthm}.
	\end{enumerate}
	
	Free noncommutative function theory (and the special case of commutative multivariable operator theory)
	arose as a way to understand the functional calculus for several noncommuting operators on a Hilbert space.
	Fundamentally, one can see the established effort as a ``complete" analog of real and complex analysis.
	(Here we use the word \emph{complete} in the sense of ``complete positivity" or ``complete boundedness" from operator theory.)
	The raw power of ``completeness" has generated significant theoretical interest
	both classically \cite{lo34, kraus36, choikraus, pau86, pau02, schoenberg1942}
	and in modern times \cite{agmcbook2, vvw12, BMV16, BMV162, po89, po91, popescuhardy, po08, muhlysolel, muhlysolel2}, as well as
	applications to complex analysis \cite{agmcbook1, agmcbook2},
	free probability \cite{voi04, voi10, pastdcauchy, pptdherglotz, jekel},
	sums of squares certificates for positivity \cite{schmu91, put93, heltonPositive, helkm12, pascoeratl},
	quantum information theory \cite{helton2017tracial, klep2008connes, pascoeisrael},
	number theory \cite{pascoeTrace},
	the theory of change of variables with an eye toward robust control \cite{helkm11, pascoemathz, plush1, plush2},
	and the theory of spectrahedra from semidefinite programming and matrix convexity \cite{heltmc12, augat2018bianalytic, kennyd, davidson2017dilations}.
	
	However, while the specifics of any given problem will eventually probably involve at least some actual operator theory or Hilbert space theoretic methods,
	for signficant parts of the ``complete" backend, the paradigm can be shifted to a general theory of endomorphisms of some object with a ``complete" structure. 
	(Operators on a Hilbert space are, of course, endomorphisms of that Hilbert space.)
	Our main study, Section \ref{weakalgsec} gives a radical new foundation for a reimagining of free noncommutative function theory for endomorphisms of a structured topological space.
	Typically, the idea in free noncommutative function theory so far has been to consider maps which preserve either direct sums and similarities or, equivalently, intertwiners.
	While, in the end, we do consider such interwiner-preserving maps, it turns out that maps which preserve invariance of structures are more paramount and have good approximation theory.
	(In the case of operators, we want to consider functions such that is $S$ is an invariant subspace for $X^{\oplus n},$ then it is invariant for $f(X)^{\oplus n}.$) 
	We also somewhat dispense with the idea of ``levels" which is pervasive in free noncommutative function theory.
\section{Weak algebraicity} \label{weakalgsec}
	\subsection{Structured topological spaces}
	A \dfn{structured topological space} $(\mathcal{H}, \mathscr{S})$ is a topological space $\mathcal{H}$
	equipped with a \dfn{structure} $\mathscr{S} \subseteq \bigcup P( \mathcal{H}^{n})$
	satisfying
	\begin{enumerate}
		\item $S \in \mathscr{S} \Rightarrow \overline{S}\in \mathscr{S},$
		%\item $S, T \in \mathscr{S} \Rightarrow S \times T \in \mathscr{S},$
		\item the block diagonal stucture $\{(v_{i_1}, \ldots, v_{i_n}) | v_i \in \mathcal{H}\}$ for each sequence $i_1, \ldots, i_n.$
	\end{enumerate}
	We say a structure is \dfn{regular} if it is closed under arbitrary intersections. We will often refer to $S \in \mathscr{S}$ as \dfn{structured}.

	\begin{example}
		Take $\mathcal{H}$ be a topological space and let $\mathscr{S}$
		be such that any for any $n \in \mathbb{N}$ and $V \subseteq \mathcal{H}^n$, we have $V \in \mathscr{S}.$ Then,
		$(\mathcal{H}, \mathscr{S})$ is a structured topological space. We call this the \dfn{discrete structure} on $\mathcal{H}.$
	\end{example}
	\begin{example}
		Take $\mathcal{H}$ be a topological group and let $\mathscr{S}$
		be such that any for any $n \in \mathbb{N}$ and $V \subseteq \mathcal{H}^n$
		a subgroup, we have $V \in \mathscr{S}.$ Then,
		$(\mathcal{H}, \mathscr{S})$ is a structured topological space. We call this the \dfn{group structure}.
	\end{example}
	\begin{example}
		Take $\mathcal{H}$ to be a topological vector space and let $\mathscr{S}$
		be such that any for any $n \in \mathbb{N}$ and $V \subseteq \mathcal{H}^n$
		a vector subspace, we have $V \in \mathscr{S}.$ Then,
		$(\mathcal{H}, \mathscr{S})$ is a structured topological space.
		We call this the \dfn{vector space structure}.
	\end{example}
	\begin{example}
		Take $\mathcal{H}$ to be a topological
		space and let $\mathscr{S}$ be the structure containing only
		the block diagonal structures. We call this the \dfn{trivial structure} on $\mathcal{H}.$
	\end{example}

	Let $(\mathcal{H},\mathscr{S})$ be a structured topological spaces.
	We let $End(\mathcal{H})$ denote the \dfn{endomorphisms of $\mathcal{H},$}
	the set of continuous maps $X: \mathcal{H} \rightarrow \mathcal{H}$ such that
	$$S \in \mathscr{S} \cap P(\mathcal{H}^{n+m+1})
	\Rightarrow (\textrm{id}_{\mathcal{H}}^{\oplus m} \oplus X\oplus \textrm{id}_{\mathcal{H}}^{\oplus n})S \in \mathscr{S},$$
	where $\textrm{id}_{\mathcal{H}}$ denotes the identity map on $\mathcal{H}.$

	Let $(\mathcal{H},\mathscr{S})$ be  a structured topological space.
	Let $f: (\mathcal{H}^n)^{\Lambda} \rightarrow \mathcal{H}^n$
	be a function such that, for any $S \in \mathscr{S}\cap P(\mathcal{H}^n),$
	if $a=(a_\lambda)_{\lambda\in \Lambda} \in  S\in \mathscr{S}$
		then $f(a)\in S.$
	We call such an $f$ an \dfn{operation.}
	The following proposition justifies the use of the name ``endomorphism."
	\begin{proposition}
		Let $(\mathcal{H},\mathscr{S})$ be a structured topological space.
		Let $f: (\mathcal{H}^2)^{\Lambda} \rightarrow \mathcal{H}^2$ be an operation.
		For any $X\in End(\mathcal{H}),$ and $(a_\lambda)_{\lambda\in \Lambda}\in (\textrm{Graph } X)^{\Lambda},$  we have that $f(a) \in \textrm{Graph }X.$
	\end{proposition}
	\begin{proof}
		Consider the block diagonal structure $S=\{ (h,h) | h\in\mathcal{H}\}.$
		Note that $\textrm{Graph }X= (\textrm{id}_{\mathcal{H}} \oplus X) S \in \mathscr{S}.$
		Suppose $(a_\lambda)_{\lambda\in \Lambda} \in \textrm{Graph }X.$
		Because $\textrm{Graph }X$ is structured, we have that $f(a)\in \textrm{Graph }X.$ 
	\end{proof}
	\begin{example}
		Let $(\mathcal{H},\mathscr{S})$ be  a topological group equipped with the group structure.
		Consider the operation group multiplication $f((h_1,k_1),(h_2,k_2)) = (h_1h_2,k_1k_2).$
		So, for $X \in End(\mathcal{H}),$ $X(h_1h_2) = Xh_1 Xh_2$ by the preceeding proposition. That is, $X$ is a group homomorphism.
	\end{example}
	\begin{example}
		Let $(\mathcal{H},\mathscr{S})$ be a topological vector space equipped with the vector space structure.
		Consider the operations vector addition $f((h_1,k_1),(h_2,k_2)) = (h_1+h_2,k_1+k_2)$ and $g_z((h,k))= (zh,zk)$ where $z$ is an element of the ground field.
		So, for $X \in End(\mathcal{H}),$ $X(h_1+h_2) = Xh_1 + Xh_2$ and  $Xzh=zXh$ by the preceeding proposition. That is, $X$ is a linear operator.
	\end{example}
	\begin{example}
		Let $(\mathcal{H},\mathscr{S})$ be a topological vector space over $\mathbb{R}$
		equipped with a structure $\mathscr{S}$ consisting of convex cones.
		Consider the operations vector addition $f((h_1,k_1),(h_2,k_2)) = (h_1+h_2,k_1+k_2)$ and $g_z((h,k))= (th,tk)$ where $t\geq 0$.
		So, for $X \in End(\mathcal{H}),$ $X(h_1+h_2) = Xh_1 + Xh_2$ and  $Xth=tXh$ by the preceeding proposition. Note that, $X$ is a linear operator
		since $Xh + X(-h) = 0.$
	\end{example}
	\begin{example}
		Let $(\mathcal{H},\mathscr{S})$ be a topological vector space over $\mathbb{R}$
		equipped with a structure $\mathscr{S}$ consisting of convex sets.
		Consider the convex combination operations  $f_t((h_1,k_1),(h_2,k_2)) = (th_1+(1-t)h_2,tk_1+(1-t)k_2)$ where $t \in [0,1].$
		So, for $X \in End(\mathcal{H}),$ $X(th_1+(1-t)h_2) = tXh_1 + (1-t)Xh_2$ by the preceeding proposition. Therefore, we see that
		$X$ is affine linear.
	\end{example}
	Generally, we see that algebraic structures are preserved by endomorphisms, \emph{even if we do not know what the operations are.}

	\begin{example}
		Let $(\mathcal{H},\mathscr{S})$ be a topological space equipped with the trivial structure.
		The set $End(\mathcal{H})$ consists of only the identity element, as the graph of an endomorphism must be in the structure.
	\end{example}
	\begin{example}
		Let $(\mathcal{H},\mathscr{S})$ be a topological space equipped with the discrete structure.
		The set $End(\mathcal{H})$ consists of all contiuous maps from $\mathcal{H}$ to $\mathcal{H}$
	\end{example}
	Equip $End(\mathcal{H})$ with the topology of pointwise convergence. (This is analogous to strong operator topology convergence.)
	A tuple $(\mathcal{T}, \mathcal{H}, \Lambda)$ is called a \dfn{space over $\mathcal{H}$} if $\mathcal{H}$ is a topological space and $\Lambda$
	is a collection of maps $\lambda: \mathcal{T}\rightarrow End(\mathcal{H})$ such the maps $\lambda$ separate points in $\mathcal{T}.$ We naturally endow 
	the set $\mathcal{T}$ with the weak topology.
	\begin{example}
		Take $\mathcal{H}$ to be a Hilbert space. Let $\mathcal{M}, \mathcal{N}$ also be Hilbert spaces.
		Let $\mathbf{B}(\mathcal{M} \otimes \mathcal{H}, \mathcal{N} \otimes \mathcal{H})$ denote the bounded linear operators
		from $\mathcal{M} \otimes \mathcal{H}$ to $\mathcal{N} \otimes \mathcal{H}.$
		Let $\mathcal{T} \subseteq \mathbf{B}(\mathcal{M} \otimes \mathcal{H}, \mathcal{N} \otimes \mathcal{H}).$
		For $\mu \in \mathcal{M}, \nu \in \mathcal{N}$ define $\lambda_{\mu,\nu}(X) = (\nu\otimes I_{\mathcal{H}})X(\mu\otimes I_{\mathcal{H}})$
		where $I_{\mathcal{H}}$ denotes the identity operator.
		Let $\Lambda$ be the collection of all such $\lambda_{\mu, \nu}.$
		Then, $(\mathcal{T}, \mathcal{H}, \Lambda)$ is a space over $\mathcal{H}.$
	\end{example}

	\begin{example}
	Take $\mathcal{H}$ to be a topological space.
	Then, $(End(\mathcal{H}), \mathcal{H}, \{\mathrm{id}_{\mathcal{H}}\})$ is a space over $\mathcal{H}.$ 
	\end{example}
	Given $\mathcal{H}$ a topological space, let $Fun(\mathcal{H})$ denote the \dfn{continuous functions from $\mathcal{H}$ to $\mathcal{H}.$}
	We endow $Fun(\mathcal{H})$ with the topology of pointwise convergence.
	Let $\mathbb{X} \subseteq Fun(\mathcal{H})^n.$
	We say $S \subseteq \mathcal{H}^{nm}$ is \dfn{jointly invariant} for $\mathbb{X}$ if $X^{\oplus m} S \subseteq S$ for all $X \in \mathbb{X}.$
	Given $(\mathcal{H},\mathscr{S})$ a structured topological space, $(\mathcal{T}, \mathcal{H}, \Lambda)$ a space over $\mathcal{H}$ and $X_1,\ldots, X_n \in \mathcal{T},$
	we say $S\subseteq \mathcal{H}^{nm}$ is \dfn{jointly invariant} for $\oplus^n_{i=1} X_i$ if $S$ is jointly invariant for the set
	$\{\oplus^n_{i=1} \lambda(X_i) | \lambda \in \Lambda\}.$
	\begin{example}
	Let $\mathcal{H}$ be a topological vector space.
	Take $(\mathcal{H},\mathscr{S})$ be the structured topological space vector space with structure coming from the vector space structure on $\mathcal{H}.$
	Note
	$(End(\mathcal{H}), \mathcal{H}, \{\mathrm{id}_{\mathcal{H}}\})$ is a space over $\mathcal{H}.$ 
	Given $X \in End(\mathcal{H}), \mathcal{H},$ an invariant subspace $S$ for $X^{\oplus n}$ gives an $S$ which is jointly invariant for $X.$
	\end{example}
	
	\subsection{Equidynamical approximation}
	Let $\mathbb{X} \subseteq Fun(\mathcal{H})^n.$ We define the \dfn{equi-invariants of $\mathbb{X}$,} denoted $\mathbb{X}\DC$
	to be the set of $Y \in Fun(\mathcal{H})^n$ such that if a closed set $S$ is jointly invariant for $\mathbb{X}$ then $S$ is jointly invariant for $Y.$
	Let $(\mathcal{T}, \mathcal{H}, \Lambda)$ a space over $(\mathcal{H},\mathscr{S}).$ 
	Given $X \in \mathcal{T}^n,$ we define the set $X\DC,$ the \dfn{equi-invariants of $X$,} denoted $X\DC$
	to be the set of equi-invariants for the set of equi-invariants for the set $\{\oplus^n_{i=1} \lambda(X_i) | \lambda \in \Lambda\}.$

	\begin{theorem}[Equidynamical approximation] \label{equidyn}
		Let $(\mathcal{H},\mathscr{S})$ be a structured topological space.
		Suppose $\mathscr{B}\subseteq End(\mathcal{H})$ contains the identity and is closed under composition.
		Moreover suppose that for every $v \in \mathcal{H}^n,$ the set $\{ X^{\oplus n}v | X \in \mathscr{B}\}\in \mathscr{S}.$
		Then, $\overline{\mathscr{B}} = \mathscr{B}\DC.$
		
		Namely, if $End(\mathcal{H})$ is closed in $Fun(\mathcal{H}),$  $\mathscr{B}\DC \subseteq End(\mathcal{H}).$
	\end{theorem}
	\begin{proof}
		Let $Y \in \mathscr{B}\DC.$
		Pick a basic neighborhood of $Y.$ That is, fix $v_1, \ldots, v_n \in \mathcal{H}$ and neighborhoods $U_1,\ldots U_n$ of $Yv_1, \ldots, Yv_n$ and consider 
			$B= \{Z \in End(\mathcal{H})| Zv_i \in U_i\}.$
		Let $v = \oplus^n_{i=1} v_i.$
		Let $$S = \{ X^{\oplus n}v | X \in \mathscr{B}\}$$
		Note $S$ is jointly invariant for $\mathscr{B}.$
		Therefore, $Y^{\oplus n}v \in \overline{S}.$ So there is an $X \in \mathscr{B}$ such that $X^{\oplus n}v \in U_1 \times \ldots \times U_n.$
		That is, $Xv_i \in U_i.$ Thus, $B \cap \overline{\mathscr{B}} \neq \emptyset.$
	\end{proof}
	\begin{example}
		A continuous function $Y:\mathcal{H}\rightarrow \mathcal{H}$
		on some topological vector space is in the strong closure of the algebra generated by some operators $X_1,\ldots, X_n$ 
		if and only if any closed joint invariant subspace for the $X_i^{\oplus n}$ is invariant for $Y^{\oplus n}.$
		Here $\mathscr{B}$ in the equidynamical approximation theorem is the algebra generated by the $X_i.$

		One may view equidynamical approximation
		as an analogue of the von Neumann double commutant theorem. (In fact, one may use it to prove the double commutant theorem in conjunction with the spectral
		theorem.)
	\end{example}
	A key point is that the ``trivial invariants" with respect to a given structure characterize the endomorphisms of that structure.
	\begin{corollary}
		Let $(\mathcal{H},\mathscr{S})$ be a structured topological space.
		Suppose $End(\mathcal{H})$ contains the identity and is closed under composition.
		Suppose $End(\mathcal{H})$ is closed in $Fun(\mathcal{H}).$
		Moreover suppose that for every $v \in \mathcal{H}^n,$ the set $\{ X^{\oplus n}v | X \in End(\mathcal{H})\}\in \mathscr{S}.$
		Let $Y \in Fun(\mathcal{H}).$ If every closed $S\in\mathscr{S}$ which is jointly invariant for $End(\mathcal{H})$ is jointly invariant for $Y,$
		then $Y \in End(\mathcal{H}).$
	\end{corollary}

	\subsection{Invariant structure preserving functions}
	
	Let $(\mathcal{H}, \mathscr{S})$ be a structured topological space. Let $(\mathcal{T}, \mathcal{H}, \Lambda),$ $(\mathcal{R}, \mathcal{H}, \Omega)$ be spaces over $\mathcal{H}.$
	We say $f: \mathcal{T} \rightarrow \mathcal{R}$ is \dfn{invariant structure preserving} if for every $n\in \mathbb{N},$ for any $X_1,\ldots, X_n \in \mathcal{T}$ and closed structured $S$ jointly invariant for $\oplus^n_{i=1} X_i$,
	we have that $S$ is jointly invariant for $\oplus^n_{i=1} f(X_i).$
	We denote the space of invariance preserving functions from $\mathcal{T}$ to $End(\mathcal{H})$ by $H(\mathcal{T}).$

	\begin{example}
	Let $\mathcal{H} = \ell^2(\mathbb{Z})$ be equipped with the vector space structure. Note that $End(\mathcal{H})=\BH.$
	Let $\mathcal{T}$ be the set of invertible elements in $\BH$ as a space over $\mathcal{H}.$
	The function $h(X) = X^2$ is an invariant structure preserving function. However, the function $f(X)=X^{-1},$ fails to be invariant structure preserving as $\ell^2(\mathbb{N})$ is invariant for the bilateral shift but not its inverse.
	\end{example}

	We say that $\mathscr{Z}$ is a \dfn{compatiable
	structure} on $H(\mathcal{T})$ if:
	\begin{enumerate}
		\item $(H(\mathcal{T}), \mathscr{Z})$ is structured topological space,
		\item Every  $Z\in\mathscr{Z}$ is closed under pointwise composition,	
		\item for every $Z\in \mathcal{Z}$ such that $Z \subseteq H(\mathcal{T})^n,$ we have that $\textrm{id}_{\mathcal{H}^n} \in Z,$ 
		\item for any $Z \in \mathscr{Z},$ $X_1, \ldots, X_n \in \mathcal{T},$ and $v_1,\dots, v_n \in \mathcal{H},$ the set $\{\bigoplus^n_{i=1} f_i(X_{i})v_{i} | \oplus^n_{i=1} f_i \in Z\}\in\mathscr{S}.$
	\end{enumerate}
	A \dfn{structure basis} is a structured set of functions $\mathscr{B}$ which contains the maps $\lambda.$
	If the compatiable structure $\mathscr{Z}$ is regular, then there is a unique smallest structure basis. We denote the unique minimal structure basis by $P(\mathcal{T}).$
	In the case where the structure $\mathscr{Z}$ is not regular, we abuse notation and say that $P(\mathcal{T})$ is some structure basis.
	\begin{example}
		Let $(\mathcal{T}, \mathcal{H}, \Lambda)$ a space over $(\mathcal{H},\mathscr{S})$ a topological vector space with the vector space structure. 
		A compatiable structure arises naturally from the vector space structure on $End(\mathcal{H}).$
		A structure basis is given by polynomials in $\lambda.$
	\end{example}
	\begin{example}
		Let $(\mathcal{T}, \mathcal{H}, \Lambda)$ a space over $(\mathcal{H},\mathscr{S})$ a topological space with the trivial structure. 
		The set $End(\mathcal{H})$ has one element and so the structure basis would contain only a constant function equal to the identity map.
		Unless $\mathcal{H}$ is a singleton, the set $\{\textrm{id}_{\mathcal{H}}v\}$ is not structured. Therefore if $|\mathcal{H}|>1$ no structure basis exists.
	\end{example}
	\begin{example}
		Let $(\mathcal{T}, \mathcal{H}, \Lambda)$ a space over $(\mathcal{H},\mathscr{S})$ a topological space with the discrete structure. 
		A compatiable structure is given by the discrete structure.
		A structure basis is given by words in the $\lambda.$
	\end{example}
	\begin{example}
		Let $(\{\textrm{id}_{S_3}\}, \mathcal{H}, \Lambda)$ a space over $(S_3,\mathscr{S})$ the symmetric group $S_3$ with the group structure. 
		We certainly need the identity map and the homomorphism $O$ taking everything to the identity element $e.$
		Let $g$ be an element of order $3.$
		Therefore, we need an automorphism $X$ taking $g$ to $g^{-1},$ so that $\{ \textrm{id}_{S_3}g, e, Xg\}$ would be structured.
		Note that such an automorphism must be of the form $x \mapsto pxp$ for some $p$ an element of order $2.$
		Note that such an automorphism does not fix the subgroup $\{gp, e\}.$
		Therefore, no structure basis exists.
	\end{example}
	\begin{example} \label{coneexample}
		Let $(\mathcal{H},\mathscr{S})$ be a topological vector space over $\mathbb{R}$
		equipped with a structure $\mathscr{S}$ consisting of convex cones.
		A structure basis for a space $(\mathcal{T}, \mathcal{H}, \Lambda)$ over $\mathcal{H}$
		would be the set of polynomials in the $\lambda$ with non-negative coefficients.
	\end{example}
	\begin{example} \label{convexexample}
		Let $(\mathcal{H},\mathscr{S})$ be a topological vector space over $\mathbb{R}$
		equipped with a structure $\mathscr{S}$ consisting of convex sets.
		A structure basis for a space $(\mathcal{T}, \mathcal{H}, \Lambda)$ over $\mathcal{H}$
		would be the set of polynomials in the $\lambda$ with non-negative coefficients summing to $1.$
	\end{example}
	Note that if the structure is finer, generally less maps are needed in the structure basis, given that we are dealing with endomorphisms
	which are the same as functions. For example, we needed less maps for the discrete structure than for a vector space with the cone structure than
	for a vector space with the vector space structure.

	\subsection{Approximation from a structure basis}

	A similar argument to that for equidynamical approximation gives that a structure basis approximates any invariant structure preserving function.
	\begin{theorem}\label{equifun}
		Let $(\mathcal{T}, \mathcal{H}, \Lambda)$ a space over a structured topological space $(\mathcal{H},\mathscr{S})$ such that the set of functions on $\mathcal{T}$
		are endowed with a compatiable structure such that a structure basis exists.
		Then,
			$$\overline{P(\mathcal{T})} = H(\mathcal{T}).$$
	\end{theorem}
	\begin{proof}
		Let $g \in H(\mathcal{T}).$
		Pick a basic neighborhood of $f.$ That is, 
		fix $X_1,\ldots, X_n,$ $v_1, \ldots v_n$ and and neighborhoods $U_1,\ldots U_n$ of $g(X_n)v_1, \ldots, g(X_n)v_n$ and consider 
			$B= \{Z \in End(\mathcal{H})^n| Z_iv_i \in U_i\}.$
		Let $$S = \{ \oplus_{i=1} f(X_i)v_i | f \in P(\mathcal{T})\}.$$
		Note $S$ is structured and jointly invariant for $\oplus^n_{i=1} X_i.$
		Therefore, $\oplus^n_{i=1} g(X_i)v_i \in \overline{S}.$
		So there is an $f \in P(\mathcal{T})$ such that $\oplus^n_{i=1} f(X_i)v_i \in U_1 \times \ldots \times U_n.$
		That is, $f(X_i)v_i \in U_i.$ Thus, $B \cap \overline{P(\mathcal{T})} \neq \emptyset.$
	\end{proof}
	
	\begin{example}\label{exampleschoenberg}
		Endow $\mathbb{R}$ with the convex cone structure.
		Let $F \subseteq End(\mathbb{R}) = \mathbb{R}.$
		Suppose $|\inf F| \leq \sup F$ and $\sup F, \inf F \notin F.$
		Note that the structure basis $P(F)$ consists of positive combinations of monomials $x^n$ by Example \ref{coneexample}.
		Let $f:F\rightarrow \mathbb{R}$ be an invariant structure preserving function.
		So, $f$ is a pointwise limit of functions of the form $\sum^N_{i=0} \gamma_i x^i$ where $\gamma_i \geq 0$ by Theorem \ref{equifun}.
		Note, for functions of the form $g(x) = \sum^N_{i=0} \gamma_i x^i,$ $|g(z)| \leq f(x)$
		for all $|z|<x.$  Therefore, $f$ analytically continues to a disk of radius $\sup F$ with nonnegative power series coefficients.
	\end{example}
	
	\begin{example}\label{exampleschoenberg2}
		Endow $\mathbb{R}$ with the convex set structure.
		Let $F \subseteq \mathbb{R}.$ Identify $F$ with the subset of linear maps in $End(\mathbb{R}).$
		Suppose $|\inf F| \leq \sup F$ and $\sup F, \inf F \notin F,$ $\sup f>1.$
		Note that the structure basis $P(F)$ consists of convex combinations of monomials $x^n$ by Example \ref{convexexample}.
		Let $f:F\rightarrow \mathbb{R}$ be an invariant structure preserving function.
		So, $f$ is a pointwise limit of functions of the form $\sum^N_{i=0} \gamma_i x^i$
		where $\gamma_i \geq 0$ and $\sum^N_{i=0} \gamma_{i}=1$ by Theorem \ref{equifun}.
		Note, for functions of the form $g(x) = \sum^N_{i=0} \gamma_i x^i,$ $|g(z)| \leq f(x)$
		for all $|z|<x.$
		Therefore, $f$ analytically continues to a disk of radius $\sup F$ with nonnegative power series coefficients which sum to $1$.
		(One can easily extend this to the $ax+b$ group, which is the full endomorphism group of $\mathbb{R}$ with the convex set structure. Note
		that if $\sup F\leq 1,$ we can obtain any function with positive power series coefficients that sum to less than or equal to $1.$)
	\end{example}

	\subsection{$N$-preservers}
	
	Let $(\mathcal{H}, \mathscr{S})$ be a structured topological space. Let $(\mathcal{T}, \mathcal{H}, \Lambda),$ $(\mathcal{R}, \mathcal{H}, \Omega)$ be spaces over $\mathcal{H}.$
	We say $f: \mathcal{T} \rightarrow \mathcal{R}$ is \dfn{$N$-invariant structure preserving} if for any $X_1,\ldots, X_N \in \mathcal{T}$ and closed $S$ jointly invariant for $\oplus^N_{i=1} X_i$
	we have that $S$ is jointly invariant for $\oplus^N_{i=1} f(X_i).$ (In particular, $N$ is fixed.)
	We denote the space of invariance preserving functions from $\mathcal{T}$ to $End(\mathcal{H})$ by $H_N(\mathcal{T}).$

	We say $P(\mathcal{T})$ has the \dfn{finite exact approximation property} if  for every $X_1, \ldots, X_n \in \mathcal{T}$ and $v_1, \ldots, v_n \in \mathcal{H}$
	the set $\{ \oplus^n_{i=1} p(X_i)v_i | p \in P(\mathcal{T})\}$ is closed.
	\begin{theorem}\label{nequifun}
		Let $(\mathcal{T}, \mathcal{H}, \Lambda)$ a space over a structured topological space $(\mathcal{H},\mathscr{S})$ such that the set of functions on $\mathcal{T}$
		are endowed with a compatiable structure such that a structure basis exists.
		A function $f \in H_N(\mathcal{T})$ if and only if for every $X_1, \ldots, X_N \in \mathcal{T}$ and $v_1, \ldots, v_N \in \mathcal{H}$
		$$\bigoplus^N_{i=1}f(X_i)v_i \in \overline{ \{ \oplus^N_{i=1} p(X_i)v_i | p \in P(\mathcal{T})\}}.$$
		
		Namely, if $P(\mathcal{T})$ has the finite exact approximation property, then $f \in H_N(\mathcal{T})$ if and only if
		for every $X_1, \ldots, X_N \in \mathcal{T}$ and $v_1, \ldots, v_N \in \mathcal{H}$ there exists a $p \in P(\mathcal{T})$
		such that $f(X_i)v_i=p(X_i)v_i.$
	\end{theorem}
	\begin{proof}
		Let $f \in H_N(\mathcal{T}).$
		Let $$S_v = \{ \oplus^N_{i=1} p(X_i)v_i | h \in p(\mathcal{T})\}.$$
		Note $S$ is structured and jointly invariant for $\oplus^N_{i=1} X_i.$
		Therefore, $\oplus^N_{i=1} f(X_i)v_i \in \overline{S}.$
		
		Note that any invariant structured set must be a union of sets of the form $S_v$ so therefore the converse also holds.
	\end{proof}
	
	\begin{example}\label{exampleschoenberg3}
		Endow $\mathbb{R}$ with the convex cone structure.
		Let $F \subseteq End(\mathbb{R}) = \mathbb{R}.$
		Recall that the structure basis $P(F)$ consists of positive combinations of monomials $x^n$ by Example \ref{coneexample}.
		Let $f:F\rightarrow \mathbb{R}$ be an $N$-invariant structure preserving function.
		So, for every $x_1, \ldots, x_N \in \mathbb{R}$ there exists a sequence of polynomials $p_k(x) = \sum^k_{m=0} \gamma_{m,k} x^m$ where $\gamma_{m,k}\geq 0$
		such that $f(x_i) = \lim_{k\rightarrow \infty} p_k(x_i).$
		
		If, in addition, $F \subseteq \mathbb{R}^{\geq 0},$ $x_1\leq \ldots \leq x_N,$
		we see that $f(x_i) = \sum^{\infty}_{m=0} \gamma_mx_i^m + \gamma_{\infty}\chi_{x_i=x_N}.$ That is, at any $N$ points $f$ agrees with an invariant structure
		preserving function plus an indicator function at the largest point. (In fact, this is an invariant structure preserving function on $\{x_1,\ldots x_n\}.$)
	\end{example}

	\subsection{Free functions}
	
	We let $Hom(\mathcal{H}^n,\mathcal{H}^m)$ denote the \dfn{homomorphisms from $\mathcal{H}^n$ to $\mathcal{H}^m,$}
	the set of continuous maps $X: \mathcal{H}^n \rightarrow \mathcal{H}^m$ such that
	$$S \in \mathscr{S} \cap P(\mathcal{H}^{k+l+n})
	\Rightarrow (\textrm{id}_{\mathcal{H}}^{\oplus k} \oplus X\oplus \textrm{id}_{\mathcal{H}}^{\oplus l})S \in \mathscr{S}.$$

	Let $(\mathcal{H}, \mathscr{S})$ be a structured Hausdorff topological space.
	Let $(\mathcal{T}, \mathcal{H}, \Lambda),$ $(\mathcal{R}, \mathcal{H}, \Omega)$ be spaces over $\mathcal{H}.$
	We say $f: \mathcal{T} \rightarrow \mathcal{R}$ is a \dfn{free function} if for any $X_1,\ldots, X_n, Y_1, \ldots, Y_m \in \mathcal{T}$  
	and $\Gamma \in Hom(\mathcal{H}^n,\mathcal{H}^m)$ such that if $\Gamma \oplus^n_{i=1} \lambda(X_i) = \oplus^m_{i=1}\lambda(Y_i) \Gamma$
	for all $\lambda \in \Lambda,$ then $\Gamma \oplus^n_{i=1} \omega(f(X_i)) = \oplus^m_{i=1}\omega(f(Y_i)) \Gamma$ for all $\omega \in \Omega.$
	We denote the space of free functions from $\mathcal{T}$ to $End(\mathcal{H})$ by $F(\mathcal{T}).$
	\begin{example}
	Let $\mathcal{T}$ be the set of invertible elements in $\BH.$ The function $h(X) = X^2$ is a invariant structure preserving and free function. However, the function $f(X)=X^{-1},$ while a free function, fails to be invariant structure preserving.
	\end{example}
	
	We note that given a space $\mathcal{H}$ with different structures which have the same set of homomorphisms, we obtain the same set of of free functions, but may
	obtain different sets of invariant structure preserving functions. (For example, in the case of the cone structure and the vector space structure.)
	
	\begin{proposition}
	Let $(\mathcal{H}, \mathscr{S})$ be a structured Hausdorff topological space.
	Let $(\mathcal{T}, \mathcal{H}, \Lambda),$ $(\mathcal{R}, \mathcal{H}, \Omega)$ be spaces over $\mathcal{H}.$
	Any invariant structure preserving function $f: \mathcal{T} \rightarrow \mathcal{R}$ is free.
	
	Namely,
		$$H(\mathcal{T}) \subseteq F(\mathcal{T})$$
	\end{proposition}
	\begin{proof}
		Suppose $f$ is invariant structure preserving.
		Suppose $\Gamma \oplus^n_{i=1} \lambda(X_i) = \oplus^m_{i=1}\lambda(Y_i) \Gamma.$
		Write $S = \{ (v, \Gamma v) | v\in \mathcal{H}^n\}.$
		Note that $$\oplus^n_{i=1} \lambda(X_i) \oplus \oplus^m_{i=1}\lambda(Y_i) (v, \Gamma v)=
		(\oplus^n_{i=1} \lambda(X_i)v, \Gamma \oplus^n_{i=1} \lambda(X_i) v).$$
		Therefore, $S$ is jointly invariant for $\oplus^n_{i=1} \lambda(X_i) \oplus \oplus^m_{i=1}\lambda(Y_i).$
		Moreover, $S$ is closed because $\mathcal{H}$ is Hausdorff.
		So, because $f$ is invariance preserving,
			$$\oplus^n_{i=1} \omega(f(X_i)) \oplus \oplus^m_{i=1}\omega(f(Y_i)) (v, \Gamma v)=
			(\oplus^n_{i=1} \omega(f(X_i))v, \Gamma \oplus^n_{i=1} \omega(f(X_i)) v).$$
		So we are done.
	\end{proof}
	
	We say Let $(\mathcal{H}, \mathscr{S})$ be a structured Hausdorff topological space.
	Let $(\mathcal{T}, \mathcal{H}, \Lambda)$ be a space over $\mathcal{H}.$
	We say $\mathcal{T}$ is \dfn{full} if for any $X_1, \ldots, X_n\in \mathcal{T},$ and closed joint invariant $S$ for $\oplus^n_{i=1} X_i$ 
	there exists a $Y_1, \ldots, Y_m \in \mathcal{T}$ and $\Gamma \in Hom(\mathcal{H}^m, \mathcal{H}^n)$ such that the range of $\Gamma$
	is dense in $S$ and $\oplus^n_{i=1} \lambda(X_i)\Gamma = \Gamma \oplus^m_{i=1} \lambda(Y_i).$
	Note that on a full set, free functions are invariant structure preserving.
	\begin{proposition}
		Let $(\mathcal{H}, \mathscr{S})$ be a structured Hausdorff topological space.
	Let $(\mathcal{T}, \mathcal{H}, \Lambda),$ $(\mathcal{R}, \mathcal{H}, \Omega)$ be spaces over $\mathcal{H}.$
	Suppose $\mathcal{T}$ is full. Any free function $f: \mathcal{T} \rightarrow \mathcal{R}$ is invariant structure preserving.
	
	Namely,
		$$H(\mathcal{T}) = F(\mathcal{T})$$
	\end{proposition}
	We leave the proof to the reader.
	
	Let $\Sigma$ be a directed set. We say a topological space $\mathcal{X}$ has type $\Sigma$ if for every $x\in \mathcal{X}$
	there exists a neighborhood basis of open sets $B^x_\sigma$ indexed by $\Sigma$ such that if $\sigma' \geq \sigma$ then $B_{\sigma'}\subseteq B_{\sigma}.$
	\begin{example}
		Any first countable space is type $\mathbb{N}.$
	\end{example}
	Let $\Theta$ be a directed set, $\mathcal{X}$ a topological space, we say $\|\cdot\|: \mathcal{X} \rightarrow \Theta$ is a \dfn{value}.
	We say a set $S$ in $\mathcal{X}$ is \dfn{unbounded} if $\|S\|$ is unbounded in $\Theta$.
	\begin{example}
		Any norm on a Banach space is a value, where $\Theta = \mathbb{R}^{\geq 0}.$
	\end{example}
	
	Let $(\mathcal{T}, \mathcal{H}, \Lambda),$ be a space over $\mathcal{H}.$
	Equip $\mathcal{T}$ with a topology which is finer than the strong topology which is type $\Sigma.$
	Let $\|\cdot\|$ be a value on $End(\mathcal{H}).$
	We say $\mathcal{T}$ is \dfn{dilatory} if for any net $(X_\sigma)_{\sigma \in \Sigma}$ converging to some $X$ such that $X_\sigma \in B^X_\sigma,$
	there are $\hat{X}_1,\ldots, \hat{X_j} \in End(X)$ and $\Gamma_\sigma, \Gamma_\sigma'$ homomorphisms such that
	\begin{enumerate}
		\item $(\oplus \lambda(\hat{X_j})) \Gamma_\sigma = \Gamma_\sigma   \lambda(X_\sigma),$
		\item $\Gamma_\sigma'\Gamma_\sigma = \textrm{id}_{\mathcal{H}},$
		\item $\|\Gamma_\sigma' \oplus Y_i \Gamma_\sigma\| \leq \max \|Y_i\|$ for all $Y_1, \ldots, Y_j \in End(\mathcal{H})$
	\end{enumerate}
	\begin{example}
		Let $\mathcal{T}$ be an open polynomially convex set of operators on an infinite dimensional Hilbert space, the $\mathcal{T}$ is dilatory.
	\end{example}
	The definition of dilatory is set up so that any function must be locally bounded.
	\begin{proposition}
		Let $(\mathcal{H}, \mathscr{S})$ be a structured Hausdorff topological space.
		Let $(\mathcal{T}, \mathcal{H}, \Lambda),$ be a space over $\mathcal{H}.$
		Suppose $\mathcal{T}$ is dilatory. Any free function $f: \mathcal{T} \rightarrow End(\mathcal{H})$ is locally bounded.
	\end{proposition}
	The proof is essentially the definition. We leave the details to the reader.
	For free functions on large enough domains, including polynomially convex domains, or more generally open ``free sets," locally bounded (in the norm topology) functions 
	must be analytic \cite{AGMCOP}.
	
	\begin{example}
		Let $\mathcal{H}$ be Banach space such that there is a unitary operator $\Gamma: \mathcal{H} \rightarrow \bigoplus^{\infty}_{i=1} \mathcal{H}.$
		Let $\mathcal{T} \subseteq End(\mathcal{H})^\Lambda$ for some index set $\Lambda.$
		Endow $\mathcal{T}$ with the norm topology.
		Suppose $\mathcal{T} = \bigcup^{\infty}_{n=1} \mathcal{T}_n$ such that:
		\begin{enumerate}
			\item for any sequence $(X_i)^{\infty}_{i=1}$ of points in $\mathcal{T}_n$ we have that $\Gamma^{-1}\oplus^\infty_{i=1}X_i\Gamma \in \mathcal{T}_n,$
			\item the closure of $\mathcal{T}_n$ is contained in the interior of $\mathcal{T}_{n+1}.$
		\end{enumerate}
		Then, $\mathcal{T}$ is dilatory.
		
		So, for example every free function on $\mathbf{B}(\mathcal{H})^d$ is locally bounded trivially. Similarly, for any open polynomially convex
		free noncommutative set, every function is locally bounded.
	\end{example}

	\section{Case study: $\overline{\mathrm{ball}}(\ell^p_A)$ as preservers of invariant $p$-balanced sets for $p\leq 1$}

	Let $\ell^p_A$ be the set of analytic functions on the disk such that the power series coefficients at $0$ are in $\ell^p.$
	In the case where $p=1$, we have that $\ell^p_A$ is the analytic Wiener algebra. 
	The spaces $\ell^p_A$ for $p \in [1,\infty)$ are of course exciting for a variety of reasons \cite{cheng}.
	
	We say $S\subseteq \mathbb{C}^n$ is \dfn{$p$-balanced} if for every sequence $(s_i)^{N}_{i=0}$ of elements in $S$ and sequence $c=(c_i)^{N}_{i=0}$
	such that $\|c\|_p\leq 1,$ we have that $\sum^{N}_{i=0} c_is_i \in S.$
	We define the \dfn{$p$-balanced structure} on $\mathbb{C}$ to be the stucture given by all
	$p$-balanced sets. For example, if $p=1,$ we get exactly the classically balanced sets.
	Note that the elements of $\textrm{End}(\mathcal{C})$ are naturally identified with elements of $\mathbb{C}.$
	
	Let $D\subset \mathbb{C}.$ The unique minimal structure basis $P(\mathbb{D})$ for $H(D)$ is given by functions 
	of the form $p(z) = \sum^{N}_{i=0} c_iz^i$ where $\|c\|_p\leq 1.$ (This is not true if $p>1.$)
	
	A $p$-balanced analytic function space $\mathcal{K}$ on $\mathbb{D}$ is a topological vector space of analytic functions such that
	\begin{enumerate}
	\item $\mathcal{K}$  contains all functions in $\ell^p_A$ as functions,
	\item all elements of $\ell^p_A$ are multipliers of $\mathcal{K}$,
	\item $\overline{\mathrm{ball}}(\ell^p_A)$ is closed in $\mathcal{K}.$
	\end{enumerate}
	
	\begin{theorem}
		Let $p\leq 1.$
		Let $f: \mathcal{D} \rightarrow \mathcal{C}.$
		The following are equivalent:
		\begin{enumerate}
			\item $f \in \overline{\mathrm{ball}}(\ell^p_A),$
			\item $f \in H(\mathbb{D}),$ where $\mathbb{D}$ is viewed as a subset of $\textrm{End}(\mathbb{C})$ where $\mathbb{C}$ is equipped with the
			$p$-balanced structure,
			\item For every $p$-balanced analytic function space $\mathcal{K}$ on $\mathbb{D}$ such that if $S$ is a closed $p$-balanced set
			such that $zS \subseteq S,$ then $fS\subseteq S.$
		\end{enumerate}
	\end{theorem}
	\begin{proof}
		$(1 \Leftrightarrow 2)$ follows from Theorem \ref{equifun}.
		
		$(3 \Rightarrow 1)$ Let $S= \overline{\mathrm{ball}}(\ell^p_A).$ Note $zS \subseteq S.$
		Therefore, $fS\subseteq S.$ Namely, $f \in \overline{\mathrm{ball}}(\ell^p_A).$
		
		$(1\Rightarrow 3)$ Note that $f(z) = \sum^{\infty}_{i=0} c_iz^i$ where $\|c\| \leq 1.$
		So, for any $s(z) \in S,$ $f(z)s(z) = \sum^{\infty}_{i=0} c_iz^is(z) \in S.$
	\end{proof}
	
	For example, a function is in the analytic Wiener algebra if and only if for each closed balanced convex set $S$ in the Hardy space such that $zS\subseteq S$
	we have that $fS\subseteq S.$

	\section{Case study: Weak algebraicity in noncommutative function theory}

	We fix $\mathcal{H}$ a Hilbert space.
	Let $\mathcal{M}, \mathcal{N}, \mathcal{M}', \mathcal{N}'$ be Hilbert spaces.
	Let $\mathcal{T} \subseteq \mathbf{B}(\mathcal{M} \otimes \mathcal{H}, \mathcal{N} \otimes \mathcal{H})$ viewed naturally as a space over $\mathcal{H}$
	We will consider free and invariant structure preserving functions of the form
	$f: \mathcal{T} \rightarrow \mathbf{B}(\mathcal{M}' \otimes \mathcal{H}, \mathcal{N}' \otimes \mathcal{H}).$ 
	An aim of operatorial free function theory is to understand functional calculus for several noncommuting operators and its applications. (See e. g. 
	\cite{AGMCOP,agmc2,mancuso,kennyd, kensham, jkmmp}.) We will see that invariant structure preserving functions are in the pointwise weak closure
	of the polynomials and therefore free.
	Moreover on polynomially convex sets, bounded functions can be approximated by bounded polynomials.
	The case when $\mathcal{M}$ and $\mathcal{N}$ are finite dimensional corresponds to working $M$ by $N$ matrices with operator entries.
	The invariance preserving condition then says that $S$ is a joint invariant subspace for the entries of the matrix.

	The space of invariant structure preserving functions on $\mathcal{T}$ taking values in
	$\mathbf{B}(\mathcal{M}' \otimes \mathcal{H}, \mathcal{N}' \otimes \mathcal{H})$,
	denoted $H(\mathcal{T},\mathcal{M}',\mathcal{N}'),$ is a vector space and closed in the topology of weak pointwise convergence.
	The space of free functions on $\mathcal{T}$ taking values in
	$\mathbf{B}(\mathcal{M}' \otimes \mathcal{H}, \mathcal{N}' \otimes \mathcal{H})$,
	denoted $F(\mathcal{T},\mathcal{M}',\mathcal{N}'),$ is a vector space and closed in the topology of weak pointwise convergence.
	
	Moreover, a free noncommutative polynomial naturally gives rise to a function which is both free and invariant structure preserving.
	We now describe what we mean by a free noncommutative polynomial.
	Define $E(\mathcal{T}, \mathcal{M}',\mathcal{N}')$ to be the set of functions of the form $(W\otimes I_{\mathcal{H}})X(V\otimes I_{\mathcal{H}})$ where $V, W$ are bounded operators.
	We define $P(\mathcal{T}, \mathcal{M}',\mathcal{N}')$ to be span of products of the form $E_NE_{N-1}\ldots E_1$
	where $E_i \in E(\mathcal{T},\mathcal{P}_{i-1},\mathcal{P}_{i})$ where $\mathcal{P}_i$ are Hilbert spaces and $\mathcal{P}_0=\mathcal{M}'$
	and $\mathcal{P}_N=\mathcal{N}'.$
	
The following direct corollary of Theorem \ref{equifun} shows that invariant structure preserving functions can be approximated by polynomials in the weak operator topology. Importantly, the value of a invariant structure preserving function at $X$
is contained in the weak operator topology closure of the unital algebra generated by the coordinates of $X.$
\begin{corollary}\label{mainresult}
	$$\overline{P(\mathcal{T}, \mathcal{M}',\mathcal{N}')}^{WOT} = H(\mathcal{T}, \mathcal{M}',\mathcal{N}').$$
\end{corollary}

We immediately see the following corollary: invariant structure preserving functions are free functions. Moreover, they also preserve co-invariant subspaces.
\begin{corollary}
\begin{enumerate}
	\item $H(\mathcal{T},\mathcal{M}',\mathcal{N}') \subseteq F(\mathcal{T},\mathcal{M}',\mathcal{N}'),$
	\item If $f \in H(\mathcal{T},\mathcal{M}',\mathcal{N}'),$ then $(\mathcal{N}'^* \otimes S^*) f(\mathcal{T}) \subseteq (\mathcal{M}'^* \otimes S^*) $ whenever $S$ is a closed subspace of $\hilbert$ such that $(\mathcal{N}^* \otimes S^*) X \subseteq (\mathcal{M}^* \otimes S^*).$
\end{enumerate}
\end{corollary}

We say that $\hat{\mathcal{T}}$ is an \dfn{envelope} of $\mathcal{T}$ if for every element of $X \in \hat{\mathcal{T}}$ there are $Y_1,\ldots,Y_n \in \mathcal{T}$ 
and a joint invariant or coinvariant subspace $S$ of  $Y=\oplus^N_{i=1} Y_i$ and an invertible operator $A$ such that
 $X=(I_{\mathcal{N}} \otimes A^{-1})  P_SY|_S (I_{\mathcal{M}} \otimes A).$
\begin{corollary} Let $\hat{\mathcal{T}}$ be an envelope of $\mathcal{T}.$
	Any function $f \in H(\mathcal{T})$ must continue uniquely to $\hat{\mathcal{T}}.$
\end{corollary}

  For example, every invariant structure preserving function defined on the bilateral
shift extends to all contractions.

\section{An Oka-Weil Kaplansky density theorem for free functions} \label{ballsection}

\subsection{Some background}
Let $\mathcal{T} \subseteq \mathbf{B}(\mathcal{M} \otimes \mathcal{H}, \mathcal{N} \otimes \mathcal{H}).$
We define the norm of $f\in H(\mathcal{T},\mathcal{M}',\mathcal{N}')$ to be $\|f\|=\sup_{X\in \mathcal{T}} \|f(X)\|.$ We denote the set of functions with finite norm by $H^\infty(\mathcal{T},\mathcal{M}',\mathcal{N}').$
A natural question from the point of view of approximation theory, along the lines of the Kaplansky density theorem,  is:
\begin{problem}[Kaplansky density problem] Does
\beq \label{problem} \overline{\textrm{ball}(P(\mathcal{T},\mathcal{M}',\mathcal{N}'))}^{WOT} = \overline{\textrm{ball}}(H^{\infty}(\mathcal{T},\mathcal{M}',\mathcal{N}'))?\eeq
\end{problem}
For example, one might want to approximate a bounded function by a sequence of uniformly bounded polynomials, as in the Stone-Weierstrass theorem or the Oka-Weil theorem \cite{AGMCGlobal}.
We note that if $\mathcal{T}$ is a bounded set of self-adjoint tuples, then \eqref{problem} has a positive solution.
However, this is impossible in general for reasons we will now describe, even in one variable. (That is, $\mathcal{M}=\mathcal{N}=\mathbb{C}.$)
Given $X \in \mathbf{B}(\mathcal{M} \otimes \mathcal{H}, \mathcal{N} \otimes \mathcal{H}),$ let $\mathcal{T}_X = \{ X\}$. Let $\textrm{Alg}_X$ denote the unital algebra generated by the coordinates of $X.$
By Theorem \ref{mainresult}, the question for such a $\mathcal{T}_X$ reduces to the following Kaplansky density theorem type statement:
$$\overline{\textrm{ball}(\textrm{Alg}_X)}^{WOT} = \overline{\textrm{ball}}(\overline{\textrm{Alg}_X}^{WOT}),$$
which is false in general due to counterexamples obtained classically by Wogen \cite{Wogen}-- in fact there is a single operator $T \in \BH$ for which this fails.
(The author thanks Mike Jury and Nik Weaver for pointing out this reference \cite{nik}.)

Let $\delta: \mathbf{B}(\mathcal{M} \otimes \mathcal{H}, \mathcal{N} \otimes \mathcal{H}) \rightarrow \mathbf{B}(\mathcal{K} \otimes \mathcal{H})$ be a noncommutative polynomial function.
	Define the open polynomially convex set corresponding to $\delta$:
		$$B_\delta = \{X | \|\delta(X)\|<1\}.$$
	(We only deal with square $\delta$ for notational convenience.)
	We say $\delta$ is \dfn{Archimedian} if $\delta = \delta' \oplus q_C$ for some $C>0$ where $q_C = (Cx_1,\ldots,Cx_d).$
	This essentially ensures that $B_\delta$ is bounded.
	The following is a consequence of the Helton-McCullough Positivstellensatz.
	\begin{theorem}[Helton-McCullough \cite{HK04}]
		Suppose $\mathcal{M},\mathcal{N}$ are finite dimensional.
		Suppose $\delta$ is Archimedian.
		Let $p \in P(B_\delta,\mathcal{M}',\mathcal{N}').$
		If $\|p\|_{H^\infty(B_\delta)}<1,$ then 
			$$1-p^*p = \sum q_i^*q_i + \sum r_i^*(1-\delta^*\delta)r_i.$$
		for some vector valued matrix polynomials $q_i$ and $r_i.$
	\end{theorem}

	Using the lurking isometry argument \cite{bmg,AGMCGlobal}, we see that this implies that there is a function
	$g \in \overline{\textrm{ball}}(H^{\infty}(\textrm{ball}(\mathbf{B}(\mathcal{K} \otimes \mathcal{H}),\mathcal{M}',\mathcal{N}'))$ such that
	$g\circ \delta = p.$ (Here, $\textrm{ball}(\mathbf{B}( \mathcal{K} \otimes \mathcal{H})$ may be a domain in infinitely many variables.)
	Therefore, if $p_\lambda$ is a net in $\textrm{ball}(P(B_\delta),\mathcal{M}',\mathcal{N}')$ converging to $f \in \overline{\textrm{ball}}(H^{\infty}(B_\delta,\mathcal{M}',\mathcal{N}')),$ then there is $g \in \overline{\textrm{ball}}(H^{\infty}(\textrm{ball}(\mathbf{B}( \mathcal{K} \otimes \mathcal{H}),\mathcal{M}',\mathcal{N}'))$
	such that $g\circ \delta = f.$ Therefore, we have transfer function realizations for exactly the functions in  $\overline{\textrm{ball}(P(B_\delta,\mathcal{M}',\mathcal{N}'))}^{WOT},$ and therefore can approximate by a sequence of polynomials pointwise
	in norm. So we have that the Kaplansky density problem is equivalent to  an Oka-Weil type statement.
	(Our informal treatment above will be rectified later by Theorem \ref{bigthm} and Theorem \ref{biggerthm})
	\begin{proposition}\label{gooddomains}
		Suppose $\mathcal{M},\mathcal{N}$ are finite dimensional.
		Suppose $\delta$ is Archimedian.
		If $f \in \overline{\textrm{ball}(P(B_\delta,\mathcal{M}',\mathcal{N}'))}^{WOT},$ then there exists
		$g\in \overline{\textrm{ball}}(H^{\infty}(\textrm{ball}(\mathbf{B}(\mathcal{K} \otimes \mathcal{H}),\mathcal{M}',\mathcal{N}'))$ such that $f = g \circ \delta.$
		Moreover, the following are equivalent:
		\begin{enumerate}
			\item $\overline{\textrm{ball}(P(B_\delta,\mathcal{M}',\mathcal{N}'))}^{WOT} = \overline{\textrm{ball}}(H^{\infty}(B_\delta,\mathcal{M}',\mathcal{N}')),$
			\item $\overline{\textrm{ball}(P(B_\delta,\mathcal{M}',\mathcal{N}'))}^{\textrm{ptwise in norm}}
			= \overline{\textrm{ball}}(H^{\infty}(B_\delta,\mathcal{M}',\mathcal{N}')),$
			\item every function $f \in \overline{\textrm{ball}}(H^{\infty}(B_\delta,\mathcal{M}',\mathcal{N}'))$ can be written as $ f= g \circ \delta$ where
			$g\in \overline{\textrm{ball}}(H^{\infty}(\textrm{ball}(\mathbf{B}( \mathcal{K} \otimes \mathcal{H}),\mathcal{M}',\mathcal{N}'))).$
		\end{enumerate}
	\end{proposition}
A result of Agler-McCarthy and Ball-Marx-Vinnikov \cite{AGMCOP, BMV16} states that whenever there is a matrix input that makes $\delta$ contractive, then $B_\delta$ satisfies the above theorem.

We are left with the following refinement of the Kaplansky density problem as a possible obstruction to realization theory:
\begin{problem} \label{bigproblem}
Does the Kaplansky density problem hold on every open polynomially convex $B_\delta$?
\end{problem}

\subsection{Solution to Problem \ref{bigproblem}}

We now show that Problem \ref{bigproblem} has a positive solution.
First we prove a Lyapunov type condition. We were heavily inspired by theorems from the matrix case proven in \cite{BMV16,BMV162, bmg}.
\begin{theorem}\label{bigthm}
	Let $\mathcal{H},\mathcal{J}, \mathcal{K}, \mathcal{M},\mathcal{N}$ be infinite dimensional Hilbert spaces.
	Let $X \in \mathbf{B}(\mathcal{K} \otimes \mathcal{H})$ be a strict contraction.
	Let $Y \in \mathbf{B}(\mathcal{M} \otimes \mathcal{H}, \mathcal{N} \otimes \mathcal{H}).$
	Suppose that for every self-adjoint $H \in \mathbf{B}(\mathcal{H} \otimes \mathcal{J})$
	we have that if 
		$I_{\mathcal{K}} \otimes H - (X \otimes I_{\mathcal{J}})^*(I_{\mathcal{K}} \otimes H)(X \otimes I_{\mathcal{J}})$ is positive semidefinite,
	then
		$I_{\mathcal{M}} \otimes H - (Y \otimes I_{\mathcal{J}})^*(I_{\mathcal{N}} \otimes H)(Y \otimes I_{\mathcal{J}})$ is positive semidefinite.
	Then there is an auxillary Hilbert space $\mathcal{P},$ a representation $\pi:\mathbf{B}(\mathcal{K}) \rightarrow \mathbf{B}(\mathcal{P})$ and an isometry 
		$$\bbm A & B \\ C & D \ebm: \mathcal{M}\oplus \mathcal{P} \rightarrow \mathcal{N} \oplus \mathcal{P}$$
	such that
		{\tiny$$Y= (A\otimes I_\mathcal{H}) + (B\otimes I_\mathcal{H}) (\pi\otimes \mathrm{id}_{\BH})(X)(I_{\mathcal{P}\otimes\mathcal{H}}-(D\otimes I_\mathcal{H}) (\pi\otimes \mathrm{id}_{\BH})(X))^{-1}(C\otimes I_\mathcal{H}).$$}
		
\end{theorem}
\begin{proof}
	Define the the operator system spanned by elements of the form $I_{\mathcal{K}} \otimes H - X^*(I_{\mathcal{K}} \otimes H)X$ where 
	$H$ ranges over $\mathbf{B}(\mathcal{H}).$ Note that as $X$ is contractive, taking $H = I$
	witnesses that the operator system is Archimedian. (Here Archimedian means that there is an element $T$ such that for any $Z=Z^*,$ there is a $t>0$ such that $Z + tT$ is positive definite.)
	We define a map $$\varphi(I_{\mathcal{K}} \otimes H - X^*(I_{\mathcal{K}} \otimes H)X) = I_{\mathcal{M}} \otimes H - Y^*(I_{\mathcal{N}} \otimes H)Y.$$
	By our assumptions, the map $\varphi$ is completely positive.
	Therefore, by there Stinespring theorem in combination with the Arveson extension theorem,
		the is a representation $\hat{\pi}:\mathbf{B}(\mathcal{K}\otimes \mathcal{H}) \rightarrow \mathbf{B}(\mathcal{L})$
	for some Hilbert space $\mathcal{L}$ and a $V: \mathcal{M} \otimes \mathcal{H} \rightarrow \mathcal{L}$
	such that
		$$V^*\hat{\pi}(I_{\mathcal{K}} \otimes H - X^*(I_{\mathcal{K}} \otimes H)X)V = I_{\mathcal{M}} \otimes H - Y^*(I_{\mathcal{N}} \otimes H)Y.$$
	Applying the formula to $H = ww^*$ and summing (in the strong operator topology) over $w$ in an orthonormal basis to recover a perhaps simpler formula, we see that without loss of generality
	$\hat{\pi} = \pi \otimes \mathrm{id}_{\BH}$
	where $\pi: \mathbf{B}(\mathcal{K}) \rightarrow \mathbf{B}(\mathcal{P})$ and $\mathcal{L} = \mathcal{P} \otimes \mathcal{H}.$
	Now taking $H = vw^*$ and clearing minus signs, we see that 
		$$Y^*(I_{\mathcal{N}} \otimes vw^*)Y + V^*(I_\mathcal{P} \otimes vw^*)V = I_{\mathcal{M}} \otimes vw^* +
		V^*\hat{\pi}(X^*) (I_\mathcal{P} \otimes vw^*)\hat{\pi}(X)V.$$
	Therefore, there is an isometry
		$$\bbm A & B \\ C & D \ebm$$
	such that
		$$\bbm (I_{\mathcal{N}} \otimes w^*)Y \\ (I_{\mathcal{P}} \otimes w^*)V \ebm = \bbm A & B \\ C & D \ebm \bbm (I_{\mathcal{M}} \otimes w^*) \\ (I_{\mathcal{P}} \otimes w^*) \hat{\pi}(X)V \ebm.$$
	Noting that
		$$\bbm A & B \\ C & D \ebm \bbm (I_{\mathcal{M}} \otimes w^*) & 0 \\ 0 & (I_{\mathcal{P}} \otimes w^*) \ebm = 
		\bbm (I_{\mathcal{N}} \otimes w^*) & 0 \\ 0 & (I_{\mathcal{P}} \otimes w^*) \ebm \left(\bbm A & B \\ C & D \ebm \otimes I_{\mathcal{H}} \right),$$
	we see that 
		$$\bbm Y \\ V \ebm = \left(\bbm A & B \\ C & D \ebm \otimes I_{\mathcal{H}}\right) \bbm I_{\mathcal{M}\otimes \mathcal{H}}   \\  \hat{\pi}(X)V \ebm.$$
	one solves for
		{\tiny$$Y= (A\otimes I_\mathcal{H}) + (B\otimes I_\mathcal{H}) (\pi\otimes \mathrm{id}_{\BH})(X)(I_{\mathcal{P}\otimes\mathcal{H}}-(D\otimes I_\mathcal{H}) (\pi\otimes \mathrm{id}_{\BH})(X))^{-1}(C\otimes I_\mathcal{H}).$$}
\end{proof}

Let  $\mathcal{H}$ be infinite dimensional.
Now let $f \in \overline{\textrm{ball}}(H^{\infty}(B_\delta,\mathcal{M}',\mathcal{N}')).$ Pick an $X \in B_\delta.$ We will show that there exists a sequence of polynomials $p \in \overline{\textrm{ball}}(P(B_\delta,\mathcal{M}',\mathcal{N}'))$ such that
$p(X) \rightarrow f(X)$ pointwise in norm.
Note that 
$I_{\mathcal{K}} \otimes H - (\delta(X) \otimes I_{\mathcal{J}})^*(I_{\mathcal{K}} \otimes H)(\delta(X) \otimes I_{\mathcal{J}})$ is positive semidefinite,
	then
$H - (f(X) \otimes I_{\mathcal{J}})^*H(f(X) \otimes I_{\mathcal{J}})$
essentially because the function $f$ is $H^{\infty}(B_\delta,\mathcal{M}',\mathcal{N}').$ (We need that $\mathcal{H}$ is infinite dimensional to ensure that we can realize the point $(I_{\mathcal{N}} \otimes H^{1/2})(X \otimes I_{\mathcal{J}})(I_{\mathcal{M}} \otimes H^{-1/2})$ is noncanonically unitarialy similar to a point in $B_\delta.$)
Therefore,
{\tiny$$f(X)= (A\otimes I_\mathcal{H}) + (B\otimes I_\mathcal{H}) (\pi\otimes \mathrm{id}_{\BH})(\delta(X))(I_{\mathcal{P}\otimes\mathcal{H}}-(D\otimes I_\mathcal{H}) (\pi\otimes \mathrm{id}_{\BH})(\delta(X)))^{-1}(C\otimes I_\mathcal{H}).$$}
Define {\tiny$$g(Z)=(A\otimes I_\mathcal{H}) + (B\otimes I_\mathcal{H}) (\pi\otimes \mathrm{id}_{\BH})(Z)(I_{\mathcal{P}\otimes\mathcal{H}}-(D\otimes I_\mathcal{H}) (\pi\otimes \mathrm{id}_{\BH})(Z))^{-1}(C\otimes I_\mathcal{H}).$$}
Note that $g(Z)$ can be approximated by polynomials of the form
	{\tiny$$g_{N,r}(Z)=(A\otimes I_\mathcal{H}) + (B\otimes I_\mathcal{H}) r(\pi\otimes \mathrm{id}_{\BH})(X)\sum^N_{n=0}(r(D\otimes I_\mathcal{H}) (\pi\otimes \mathrm{id}_{\BH})(Z))^{n}(C\otimes I_\mathcal{H})$$}
where $r<1.$ (Namely, for each $r<1$ there is an $N_0$ such that if $N \geq N_0$ then $rg_{N,r}$ is in $\overline{\textrm{ball}}(H^{\infty}(\textrm{ball}(\mathbf{B}(\mathcal{K} \otimes \mathcal{H}),,\mathcal{M}',\mathcal{N}')).$
Therefore, we have the following result. (Note that we can approximate at finitely many points by viewing $\oplus^N_{i=1} X_i$ as a point in $B_\delta$ by some unitary conjugation.)
\begin{theorem}\label{biggerthm}
	Let  $\mathcal{H}$ be infinite dimensional.
	Let $\delta: \mathbf{B}(\mathcal{M} \otimes \mathcal{H}, \mathcal{N} \otimes \mathcal{H}) \rightarrow \mathbf{B}(\mathcal{K} \otimes \mathcal{H})$ be a noncommutative polynomial function.
	{\tiny$$
	\overline{\textrm{ball}}(H^{\infty}(B_\delta,\mathcal{M}',\mathcal{N}'))= \overline{\textrm{ball}(P(B_\delta,\mathcal{M}',\mathcal{N}'))}^{WOT}= \overline{\textrm{ball}(P(B_\delta,\mathcal{M}',\mathcal{N}'))}^{\textrm{ptwise in norm}}.$$}
	Moreover, for any $f\in H^{\infty}(B_\delta,\mathcal{M}',\mathcal{N}')$ there exists a $g \in \overline{\textrm{ball}}(H^{\infty}(\textrm{ball}(\mathbf{B}(\mathcal{K} \otimes \mathcal{H}),\mathcal{M}',\mathcal{N}'))$
	such that $g\circ \delta = f.$
\end{theorem}
Note that we require no hypotheses on the domain containing $0$ or being connected as was the case in \cite{AGMCOP}.

\bibliography{references}
\bibliographystyle{plain}

\end{document}